\documentclass[11pt]{article}

\usepackage[utf8x]{inputenc}
\usepackage[T1]{fontenc}
\usepackage[english]{babel}

\usepackage{amsfonts}
\usepackage{amssymb}
\usepackage{amsbsy}
\usepackage{amsmath}
\usepackage{amsthm}
\usepackage{ifdraft}
\usepackage{todonotes}
\usepackage{stmaryrd}
\usepackage{tikz-cd}
\usepackage{enumerate}
\usepackage{mathtools}
\usepackage{aliascnt}
\usepackage{hyperref}
\usepackage[nameinlink,capitalize]{cleveref}

\usepackage{geometry}
\hypersetup{
    colorlinks=true,
    linkcolor=blue,
    citecolor=cyan,
    frenchlinks=true
    }

\title{Preservation under reduced products in continuous logic}
\author{Ivory Fronteau}
\date{\today}

\newtheorem{theorem}{Theorem}[section]
\newtheorem{proposition}[theorem]{Proposition}
\newtheorem{lemma}[theorem]{Lemma}
\newtheorem{fact}[theorem]{Fact}
\newtheorem{corollary}[theorem]{Corollary}

\newtheorem*{theorem*}{Theorem}

\theoremstyle{definition}
\newtheorem{definition}[theorem]{Definition}
\newtheorem{notation}[theorem]{Notation}

\theoremstyle{remark}
\newtheorem{remark}[theorem]{Remark}

\AtBeginDocument{\renewcommand{\ref}[1]{\cref{#1}}}
\renewcommand{\phi}{\varphi}
\renewcommand{\epsilon}{\varepsilon}

\def\real{\mathbb R}

\def\defeq{\coloneqq}

\def\Horn{\mathrm{H}}
\def\indices#1{\llbracket #1 \rrbracket}
\def\Pal{\mathrm{P}}
\def\frechet{\mathrm{Fin}^\ast}
\def\SCP{\mathrm{SCP}}

\def\signature{L}
\def\filter{F}
\def\filterA{G}
\def\ultrafilter{U}

\def\HP{\mathrm{HP}}
\def\PP{\mathrm{PP}}
\def\BP{\mathrm{BP}}

\begin{document}

\maketitle

\begin{abstract}
We introduce a fragment of continuous logic, analogue of Palyutin formulas (or h-formulas) in classical model theory, which is preserved under reduced products in both directions. We use it to extend classical results on complete theories which are preserved under reduced product and their stability. We also characterise the set of Palyutin sentences, Palyutin theories and other related fragments in terms of their preservation properties, both in the classical setting and the metric one.
\end{abstract}

\section*{Introduction}

Reduced products are a generalisation of both ultraproducts and cartesian products of structures. The construction is analogue to ultraproducts but one uses any proper filter instead of an ultrafilter. The question of which elementary classes of structures are closed under certain reduced product constructions is natural and of interest. By Łoś' Theorem, we know for instance that any first-order sentence is preserved under ultraproducts. On the other hand, it is very easy to find examples of sentences which are not preserved under direct products, e.g.\ the one axiomatising fields.

Horn defined the set of what are now called \emph{Horn sentences}, a fragment of first-order logic which is \emph{preserved under reduced products} \cite{hornSentencesWhichAre1951} in the sense that if $\phi$ is a Horn sentence true for $\filter$-many $A_i$'s, then the reduced product $\prod_i A_i / \filter$ also satisfies $\phi$. Actually, this fragment is the largest enjoying this property as shown by Keisler and Galvin \cite{keislerReducedProductsHorn1965, galvinHornSentences1970}. However, Horn formulas are not necessarily \emph{copreserved under reduced products} in the sense that if a reduced product $\prod_i A_i / \filter$ satisfies a Horn formula $\phi$, then $\phi$ need not be true for many $A_i$'s, or even some $A_i$. For his study of categorical Horn classes, Palyutin introduced another fragment of first-order logic consisting of sentences which are both preserved and co-preserved under reduced products \cite{palyutinCategoricalHornClasses1980}. Formulas belonging to this fragment are referred in the literature as \emph{h-formulas} or \emph{Palyutin formulas}. The set of Palyutin formulas is defined as the smallest containing $\top$, $\bot$, atomic formulas and closed under conjunction, existential and universal quantifications, and the "h-operation" $(\phi, \psi) \mapsto (\exists x \; \phi) \land \forall x \; (\phi \to \psi)$. A \emph{Palyutin sentence} is a Palyutin formula  with no free variables. \ref{sec:classical case} is devoted to recall some of Palyutin's results. Along the way, we state an analogue of Keisler-Shelah's Theorem for Palyutin sentences, i.e.\ that two structures satisfy the same Palyutin sentences if and only if they have some isomorphic reduced powers. We also show constructively that any Palyutin formula is equivalent to a Horn formula, and prove preservation results for Palyutin theories and formulas. In particular, we show that Palyutin sentences are exactly those which are both preserved and copreserved under reduced products.

In order to apply model-theoretic tools to new structures with a more analytic flavour, Ben Yaacov et al.\ extended most of the classical framework of model theory to \emph{metric structures} \cite{benyaacovModelTheoryMetric2008}. We assume in this paper some familiarity with this framework but recall here the basics. Domains of metric structures are bounded and complete metric spaces, whereas relations are real-valued, bounded and uniformly continuous mappings. (Continuous, first-order) formulas are defined inductively starting from atomic formulas, using continuous connectives $\real^n \to \real$ and quantifiers $\sup$ and $\inf$. Note that this slightly differs from the convention of \cite{benyaacovModelTheoryMetric2008} since our formulas can be $\real$-valued and not only $[0,1]$-valued, but this change does not affect the theory. There is a very natural construction of ultraproducts for metric structures where in particular, for $a$ a tuple in $\prod_i A_i$, and $\phi(x)$ an atomic formula with $x$ having the same length as $a$, then $\phi^{\prod_i A_i / \ultrafilter} (a / \ultrafilter) = \lim_{i \to \ultrafilter} \phi^{A_i}(a_i)$. By Łoś' Theorem, the previous equation holds in fact for any first-order formula.

In \ref{sec:reduced products} we recall basic definitions and properties of reduced products in continuous logic as introduced by Lopes \cite{lopesReducedProductsSheaves2013}. The definition resembles the one of ultraproducts but the limit along an ultrafilter $\ultrafilter$ is replaced by the $\limsup$ along a filter $\filter$. Goldbring and Keisler introduced an analogue for Horn sentences they called \emph{conditional sentences} and essentially proved that they are characterised by the property of being preserved under reduced products \cite{goldbringContinuousSentencesPreserved2022}. For this reason, we will refer to them as (continuous) \emph{Horn sentences}. This section is also the occasion to state the preservation result of Goldbring and Keisler in terms of continuous formulas instead of conditions.

In the same spirit as what Goldbring and Keisler did for Horn formulas, we define in \ref{sec:palyutin formulas} an analogue of Palyutin formulas in the setting of continuous logic as the smallest set containing atomic formulas and such that :
\begin{itemize}
	\item if $\phi$ is a Palyutin formula and $U$ is a nondecreasing unary connective, then $U\phi$ is a Palyutin formula ;
	\item if $\phi, \psi$ are Palyutin formulas, then $\max(\phi, \psi)$ is a Palyutin formula ;
	\item if $\phi, \psi$ are Palyutin formulas, $D$ is a nonincreasing unary connective with fixed point $\Delta$ and $x$ is a variable, then $\max(\inf_x \phi, \sup_x \min(D \phi, \Delta, \psi))$ is a Palyutin formula.
\end{itemize}
The main new aspect of our analogue of Palyutin formulas is the need, each time a nonincreasing connective is used, to also introduce its fixed point, only to ensure they have the right preservation properties. As Palyutin did in the classical framework, we isolate a continuous theory we denote by $\SCP$ (for \emph{simple cover property}) consisting of the axioms
\[\sup_y \left( \inf_x \max \left( \phi,  \max_{j = 1}^n D_j \psi_j \right) - \max_{j = 1}^n \inf_x \max \left( \phi, D_j \psi_j \right) \right) \leq 0 ~,\]
for every Palyutin formulas $\phi, \psi_1, \dots, \psi_n$ and unary nonincreasing connectives $D_1,\dots,D_n$. We then prove the following characterisation of complete theories containing $\SCP$.
\begin{theorem*}
	Let $T$ be a complete theory. The following are equivalent :
	\begin{enumerate}[(i)]
		\item $T \models \SCP$ ;
		\item $T$ is preserved under reduced products ;
		\item $T$ is preserved under products reduced by the Fréchet filter on $\omega$.
	\end{enumerate}
\end{theorem*}
The key properties that allow to prove this theorem and which were already noticed by Palyutin in the classical setting are that $\SCP$ is preserved under reduced products, and that a reduced product by the Fréchet filter is always model of $\SCP$. Along the way, we also deduce a Keisler-Shelah's Theorem in this context. As Palyutin did in the classical setting, we show that a complete NIP theory $T$ which is preserved under reduced products is  stable.
\begin{theorem*}
	Let $T$ be a complete theory which is preserved under reduced products (or equivalently, such that $T \models \SCP$). The following are equivalent :
	\begin{enumerate}[(i)]
		\item $T$ is stable;
		\item For every Palyutin formula $\phi(x,y)$ such that $\phi \geq 0$,
		\[T \models \sup_{y, z} \left( \left( \sup_x \left| \phi(x,y) - \phi(x,z) \right| \right) - \inf_x \max \left( \phi(x,y), \phi(x,z) \right) \right) \leq 0 ~;\]
		\item $T$ is NIP.
	\end{enumerate}
\end{theorem*}

Finally, in \ref{sec:preservation theorems} we prove preservation theorems for Palyutin formulas,  Palyutin theories, as well as other related fragments. In particular, we prove that the sentences $\phi$ such that $\phi^{\prod_i A_i / \filter} = \limsup_{i \to \filter} \phi^{A_i}$ holds for every reduced product $\prod_i A_i / \filter$ are exactly those that can be approximated by Palyutin formulas.

The introduction of Palyutin formulas in continuous logic was part of the author's Master's
thesis \cite{fronteauProduitsReduitsLogique2023}. I would like to address all my thanks to Tomás Ibarlucía, who supervised it and helped me writing this paper.

\section*{Notations}

We recall here some of the standard notations that will be used throughout this article, but in \ref{sec:classical case} they will be understood in the sense of classical first-order logic whereas in the other sections, they need to be understood in the sense of continuous first-order logic.

The letters $M,N$ and their variants $M',N',M_i,N_i,\dots$ will be used to denote structures or metric structures and by abuse of notation, their underlying domains. The letters $\filter, \filterA$ will be used to denote filters, which will always be proper for our purposes. Sometimes we will use the letter $\ultrafilter$ in the case of ultrafilters. The Fréchet filter on $\omega$, i.e.\ the one consisting of cofinite subsets of $\omega$, will be denoted by $\frechet$. The letter $\signature$ will denote classical signatures or metric ones. If $M,N$ are two $\signature$-structures, then we write $M \equiv N$ if they have the same theory. If we are given an embedding $M \subseteq N$, then we write $M \preceq N$ if it is elementary.

\tableofcontents

\section{The classical case}
\label{sec:classical case}

In this section we fix a signature $\signature$. Let $(M_i)_{i \in I}$ be a family of $\signature$-structures and $\filter$ a filter on $I$. Consider the relation $\sim_\filter$ on $\prod_{i \in I} M_i$ defined by $a \sim_\filter b$ if $a_i = b_i$ for $\filter$-many $i$'s. We now define an $\signature$-structure $M_\filter$ on $\prod_i M_i / \sim_\filter$ by
\begin{align*}
R^{M_\filter}( (a_i)_{i \in I} / \sim_\filter) & \text{ holds if $M_i \models \phi(a_i)$ for $\filter$-many $i$'s} \\
f^{M_\filter}( (a_i)_{i \in I} / \sim_\filter) & = (f^{M_i}(a_i))_{i \in I} / \sim_\filter ~.
\end{align*}

$M_\filter$ is called \emph{the reduced product of $M$ by $\filter$}, and can also be denoted by $\prod_i M_i / \filter$. If $M_i = N$ for every $i \in I$ for some $\signature$ structure $N$, then $M_\filter$ is \emph{the reduced power of $N$ by $\filter$} and is denoted by $N^\filter$. We will not go further in details concerning the basic definitions and properties of classical reduced products, but note that Chang and Keisler \cite{changModelTheoryThird2013} dedicated Section 6.2-3 of their book to this notion.

\begin{definition}
We say that a theory $T$ is 
\begin{itemize}
	\item \emph{preserved under reduced products} if for every family $(M_i)_{i \in I}$ of $\signature$-structures and filter $\filter$ on $I$ such that $M_i \models T$ for $\filter$-many $i$'s, we have $\prod_i M_i / \filter \models T$ ;
	\item \emph{copreserved under reduced products} if for every family $(M_i)_{i \in I}$ of $\signature$-structures and filter $\filter$ on $I$ such that $\prod_i M_i / \filter \models T$, we have $M_i \models T$ for $\filter$-many $i$'s ;
	\item \emph{preserved under reduced roots} if for every $\signature$-structure $M$ and filter $\filter$ such that $M^\filter \models T$, we have $M \models T$ ;
	\item \emph{preserved under (finite) cartesian factors} if for every (finite) family $(M_i)_{i \in I}$ of $\signature$-structures such that $\prod_i M_i \models T$, we have $M_i \models T$ for every $i$.
\end{itemize}
When we say that a formula $\phi(x)$ is preserved (resp.\ copreserved, bipreserved) under reduced products, we mean that the $\signature(x)$-theory $\{\phi(x)\}$ has this property.
\end{definition}

A first class of formulas associated to reduced products is the set of \emph{Horn formulas} which is the smallest fragment of first-order logic such that
\begin{itemize}
	\item Every formula of the form $\beta_1 \land \dots \land \beta_n \rightarrow \alpha_n$ where $\alpha, \beta_1, \dots, \beta_n$ are atomic formulas is a Horn formula (such a formula is called a \emph{Horn clause}) ;
	\item The set of Horn formulas is closed under conjunction and both quantifiers $\forall$ and $\exists$.
\end{itemize}

It is known that a first-order theory is preserved under reduced products if and only if it is axiomatisable by Horn sentences. The original proof is due to Keisler assuming CH \cite{keislerReducedProductsHorn1965}, and then Galvin who used absoluteness arguments in general case \cite{galvinHornSentences1970}. A direct proof can be found in the author's master thesis \cite{fronteauProduitsReduitsLogique2023}.
\begin{theorem}
	Let $T$ be an $\signature$-theory and $M$ an $\signature$-structure. Let $T_{\Horn}$ be the set of Horn sentences $\phi$ such that $T \models \phi$. The following are equivalent :
	\begin{enumerate}[(i)]
		\item $M \models T_\Horn$.
		\item $M^\ultrafilter \simeq N_\filter$ for some ultrafilter $\ultrafilter$, filter $\filter$ and family $N$ of models of $T$.
	\end{enumerate}
\end{theorem}

With a simple application of the Compactness Theorem, one deduces that a formula is equivalent to a Horn sentence if and only if it is preserved under reduced products. However, Horn formulas are in general not copreserved under reduced products. The Horn formula $\phi = \forall x_1 \forall x_2 \exists y \; (y \neq x_1 \lor y \neq x_2)$ is one example because the structure $M = \{1,2\}$ is not a model of $\phi$ but $M \times M \models \phi$. Palyutin isolated a fragment of first-order logic which does not have this pathology \cite{palyutinCategoricalHornClasses1980}.

\begin{definition}
The set of \emph{Palyutin formulas} is the smallest set of first-order formulas such that :
\begin{itemize}
	\item It contains atomic formulas ;
	\item It is closed under conjunction, existential quantification and universal quantification ;
	\item If $\phi, \psi$ are Palyutin formulas and $x$ is any variable, then $(\exists x \; \phi) \land \forall x \; (\phi \rightarrow \psi)$ is also a Palyutin formula.
\end{itemize}
\end{definition}

Note that the closure under both quantifications in is actually redundant with the third item in the last definition. An easy induction on this definition proves the following.

\begin{proposition}[\cite{palyutinCategoricalHornClasses1980}]
	Palyutin formulas are bipreserved under reduced products.
\end{proposition}

As a consequence of Keisler-Galvin's characterisation of Horn formulas, every Palyutin formula is equivalent to a Horn formula. However, this argument is not really constructive so we provide a direct syntactic proof of this fact.

\begin{lemma}\label{lemma:simpl phi 2 psi Horn}
Let $\phi$ be a formula. Assume that for every Horn clause $\gamma$, $\phi \rightarrow \gamma$ is equivalent to a Horn formula. Then for all Horn formula $\psi$, $\phi \rightarrow \psi$ is equivalent to a Horn formula.
\end{lemma}

\begin{proof}
Take $\psi$ a Horn formula. We may assume $\psi$ is in prenex form $Q_1 y_1 \dots Q_n y_n \; \bigwedge_{i = 1}^n \gamma_i$, where $Q_i$ is either $\exists$ or $\forall$, $y_1,\dots,y_n$ are distinct and do not appear free in $\phi$, and the $\gamma_i$'s are Horn clauses. $\phi \rightarrow \psi$ is then equivalent to :
\[
	Q_1 y_1 \dots Q_n y_n \; \bigwedge_{i = 1}^n \phi \rightarrow \gamma_i ~.
\]
By assumption, $\phi \rightarrow \gamma_i$ is equivalent to a Horn formula for each $i$, we can then conclude by closure of Horn formulas by $\wedge, \exists$ and $\forall$.
\end{proof}

\begin{theorem}
Let $\phi$ be a Palyutin formula. Then :
\begin{enumerate}[(a)]
	\item $\phi$ is equivalent to a Horn formula
	\item For every Horn formula $\psi$, $\phi \rightarrow \psi$ is equivalent to a Horn formula
\end{enumerate}
\end{theorem}

\begin{proof}
We prove both statements simultaneously by induction on $\phi$. First note that by \cref{lemma:simpl phi 2 psi Horn}, we only need to prove (b) for Horn clauses.

\begin{description}
	\item[Atomic formulas] If $\phi$ is an atomic formula, $\phi$ is Horn by definition and for every Horn clause $\gamma$, the implication $\phi \rightarrow \gamma \equiv \neg \phi \vee \gamma$ is clearly a Horn clause.
	
	\item[Conjunction] Assume (a) and (b) hold for two Palyutin formulas $\phi$ and $\psi$. Clearly, (a) holds for $\phi \wedge \psi$. Moreover, for every Horn clause $\gamma$ :
	\begin{align*}
		(\phi \wedge \psi) \rightarrow \gamma & \equiv \phi \rightarrow (\psi \rightarrow \gamma)
	\end{align*}
	By inductive hypothesis for $\psi$, $\psi \rightarrow \gamma$ is equivalent to a Horn formula $\theta$. Using inductive hypothesis for $\phi$, $\phi \rightarrow \theta$ and thus $(\phi \wedge \psi) \rightarrow \gamma$ are equivalent to a Horn formula.
	
	\item[H-operation] Assume (a) and (b) hold for two Palyutin formulas $\phi$ and $\psi$. Thus, $\psi$ is equivalent to a Horn formula but then $\phi \rightarrow \psi$ is also equivalent to a Horn formula, so the formula
	\[\theta \defeq (\exists x \; \phi) \wedge (\forall x \; (\phi \rightarrow \psi)) \]
	is equivalent to a Horn formula. Now, take $\gamma$ a Horn clause. Without loss of generality, we may assume that $x$ does not appear free in $\gamma$. One can rewrite $\theta \to \gamma$ as
	\begin{align*}
		\theta \rightarrow \gamma & \equiv (\forall x \; \neg \phi) \vee (\exists x \; (\phi \wedge \neg \psi)) \vee \gamma \\
		& \equiv (\forall x \; \neg \phi) \vee (\exists x \; (\phi \wedge \neg \psi)) \vee (\exists x \; \phi \wedge \gamma) \\
		& \equiv (\forall x \; \neg \phi) \vee \exists x \; (\phi \wedge (\neg \psi \vee \gamma)) \\
		& \equiv (\exists x \; \phi) \rightarrow \exists x \; (\phi \wedge (\psi \rightarrow \gamma)) \\
		& \equiv \forall x \; [\phi \rightarrow \exists x \; (\phi \wedge (\psi \rightarrow \gamma))] ~.
	\end{align*}
	By inductive hypothesis, $\psi \rightarrow \gamma$ and $\phi$ are equivalent to Horn formulas. Thus, the formula
	\[\exists x \; (\phi \wedge (\psi \rightarrow \gamma))\]
	is equivalent to a Horn formula. Finally, by inductive hypothesis and closure of Horn formulas by $\forall$, we conclude that $\theta \rightarrow \gamma$ is equivalent to a Horn formula. \qedhere
\end{description}
\end{proof}

\begin{definition}
	Define the theory SCP to be the set of sentences of the form
	\[\forall y_1 \dots \forall y_m \; \left( \forall x_1 \dots \forall x_l \; \left[ \phi \to \bigvee_{k = 1}^n \psi_k \right] \right) \to \bigvee_{k = 1}^n  \forall x_1 \dots \forall x_l \; \left( \phi \to \psi_k \right) ~,\]
	where $n \geq 1$ and $\phi, \psi_1,\dots, \psi_n$ are Palyutin formulas.
\end{definition}

Proofs of the following facts can be found in \cite{palmgrenDirectProofThat1994}.

\begin{lemma}\label{lemma:pres SCP}
	The theory $\SCP$ is preserved under reduced products.
\end{lemma}

\begin{lemma}\label{lemma:atomless boolean algebra SCP}
	Let $(M_i)_{i \in I}$ be a family of $\signature$-structures and $\filter$ a filter on $I$. If $\{0,1\}^{\filter}$ is an atomless boolean algebra, then $M_\filter \models \SCP$. In particular, $M_{\frechet} \models \SCP$ for every sequence $(M_n)_{n < \omega}$ of $\signature$-structures.
\end{lemma}

\begin{lemma}\label{lemma:SCP boolean combination}
	For every first-order formula $\phi(x_1,\dots,x_n)$, there exists a boolean combination of Palyutin formulas $\psi(x_1,\dots,x_n)$ such that
	\[\SCP \models \forall x_1 \dots \forall x_n \; \left( \phi \leftrightarrow \psi \right) ~.\]
\end{lemma}

\begin{corollary}\label{lemma:SCP completeness}
	Suppose that $T$ is a \emph{Palyutin-complete theory} i.e.\ that for every Palyutin sentence $\phi$, either $T \models \phi$ or $T \models \neg \phi$. Then $T \cup \SCP$ is a complete theory.
\end{corollary}

There is a natural diagonal embedding $M \subseteq M^\filter$ when we identify an element $a \in M$ by the tuple $(a,a,\dots)$. With the tools given above, it is not hard to see that for models of $\SCP$, this embedding is actually elementary.

\begin{lemma}
	If $M$ is an $\signature$-structure such that $M \models \SCP$ and $\filter$ is a filter, then $M \preceq M^\filter$.
\end{lemma}

\begin{proof}
	Let $\phi(x)$ be a first-order formula and $a \in M^x$. From \cref{lemma:SCP boolean combination} there exists a boolean combination of Palyutin formulas $\psi(x)$ such that $\SCP \models \forall x \; \left( \phi \leftrightarrow \psi \right)$. Since Palyutin formulas are bi-preserved, we clearly have $M \models \psi(a)$ if and only if $M^\filter \models \psi(a)$. But according to \cref{lemma:pres SCP}, $M$ and $M^\filter$ are both models of $\SCP$ so $M \models \phi(a)$ if and only if $M^\filter \models \phi(a)$.
\end{proof}

We now write $M \equiv_{\Pal} N$ when $M$ and $N$ satisfy the same Palyutin sentences.

\begin{theorem}\label{lemma:Palyutin equivalence}
	Let $M$ and $N$ be two $\signature$-structures. The following are equivalent :
	\begin{enumerate}[(i)]
		\item $M \equiv_{\Pal} N$ ;
		\item $M^{\frechet} \equiv N^{\frechet}$ ;
		\item $M^{\frechet \otimes \ultrafilter} \simeq N^{\frechet \otimes \ultrafilter}$ for some ultrafilter $\ultrafilter$ ;
		\item $M^{\filter} \simeq N^{\filter}$ for some filter $\filter$ ;
		\item $M^{\filter} \equiv_{\Pal} N^{\filter}$ for some filter $\filter$.
	\end{enumerate}
\end{theorem}

\begin{proof}
	Suppose (i). Since Palyutin formulas are bi-preserved, $M^{\frechet} \equiv_{\Pal} M \equiv_{\Pal} N \equiv_{\Pal} N^{\frechet}$ and thus $\mathrm{Th}(M) \cap \mathrm{Th}(N)$ is Palyutin-complete. But we also have $M^{\frechet}, N^{\frechet} \models \SCP$. We can now use \cref{lemma:SCP completeness} to deduce that $M^{\frechet} \equiv N^{\frechet}$.
	
	(ii) $\Longrightarrow$ (iii) is a consequence of Keisler-Shelah's Theorem.
	
	(iii) $\Longrightarrow$ (iv) and (iv) $\Longrightarrow$ (v) are obvious.
	
	Suppose (v). Take $\filter$ such that $M^\filter \equiv_{\Pal} N^\filter$. Then $M \equiv_{\Pal}  M^\filter \equiv_{\Pal} N^\filter \equiv_{\Pal} N$.
\end{proof}

Fix now a consistent theory $T$. Define
\begin{align*}
	T_{\Pal} & \defeq \{\phi \text{ Palyutin sentence} : T \models \phi \} ~;\\
	T_{\Pal}^\ast & \defeq T_P \cup \{\neg \phi : \phi \text{ Palyutin sentence such that } T \not\models \phi \} ~.
\end{align*}
By construction, these theories are both preserved under reduced products and reduced roots, and $T_{\Pal}$ is moreover preserved under cartesian factors. Clearly, $T \models T_{\Pal}$ so $T_{\Pal}$ is consistent. The theory $T_{\Pal}^\ast$ is also consistent but one has to give another kind of argument.

\begin{lemma}\label{lemma:consistency T_P*}
	There exists a cartesian product of models of $T$ that is a model of $T_{\Pal}^\ast$. In particular, $T_{\Pal}^\ast)$ is consistent.
\end{lemma}

\begin{proof}
	Let $\Phi \defeq \{\phi : \phi \text{ Palyutin sentence}, T \not\models \phi \}$. Take for every $\phi \in \Phi$ a model $M_\phi \models T$, $M_\phi \models \neg \phi$. Define $N \defeq \prod_{\phi \in \Phi} M_\phi$. $T_{\Pal}$ is preserved under reduced products and every $M_\phi$ is a model of $T_{\Pal}$ so $N \models T_{\Pal}$. Moreover, for every $\phi \in \Phi$, $M_\phi \not\models \phi$ so by co-preservation of $\phi$, $N \not\models \phi$. Hence, $N \models T_{\Pal}^\ast$.
\end{proof}

\begin{lemma}\label{lemma:char T_P*}
	Let $M$ be an $\signature$-structure. The following are equivalent :
	\begin{enumerate}[(i)]
		\item $M \models T_{\Pal}^\ast$.
		\item There exists a family $N = (N_i)_{i \in I}$ of models of $T$ such that $M \equiv_{\Pal} \prod_i N_i$ and for every $M' \models T_{\Pal}$, $M \times M' \equiv_{\Pal} M$.
		\item Some reduced power of $M$ is isomorphic to a reduced product of models of $T$ and for every $M' \models T$, there is some filter $\filter$ such that $(M \times M')^\filter \simeq M^\filter$.
	\end{enumerate}
\end{lemma}

\begin{proof}
	Assume (i). Using \cref{lemma:consistency T_P*}, there exists a cartesian product $N_I$ of models of $T$ satisfying $T_{\Pal}^\ast$, so clearly $N_I \equiv_{\Pal} M$. Now take $M' \models T_{\Pal}$. $M \times M' \models T_{\Pal}$ by preservation of $T_{\Pal}$. Moreover, for any Palyutin sentence $\phi$ such that $T \not\models \phi$, $M \models \neg \phi$ so $M \times M' \models \neg \phi$ by co-preservation. Hence, $M \times M' \models T_{\Pal}^\ast$ i.e.\ $M \times M' \equiv_{\Pal} M$.
	
	(ii) $\Longrightarrow$ (iii) can be seen as a consequence of \cref{lemma:Palyutin equivalence}.
	
	Assume (iii). Consider a reduced product $N_\filter$ of models of $T$ and a filter $\filterA$ such that $M^\filterA \simeq N_\filter$. This isomorphism gives us $M^\filterA \models T_{\Pal}$ and $M \models T_{\Pal}$ using co-preservation. Now, let $\phi$ be a Palyutin sentence such that $T \not\models \phi$. Let $M' \models T$ be such that $M' \models \neg \phi$. Then, $M \times M' \models \neg \phi$. If $\filter$ is now a filter such that $(M \times M')^\filter \simeq M^\filter$, then $(M \times M')^\filter, M^\filter$ and $M$ are models of $\neg \phi$, hence $M \models T_{\Pal}^\ast$
\end{proof}

\begin{theorem}\label{thm:char T_P}
	Let $M$ be an $\signature$-structure. The following are equivalent :
	\begin{enumerate}[(i)]
		\item $M \models T_{\Pal}$
		\item There exists a family $N = (N_i)_{i \in I}$ of models of $T$ such that $M \times N_I \equiv_{\Pal} N_I$.
		\item There exist some filters $\filter, \filterA$, some family $N$ of models of $T$ and some $\signature$-structure $M'$ such that $M^\filter \times M' \simeq N_\filterA$.
	\end{enumerate}
\end{theorem}

\begin{proof}
	Assume (i). Use \cref{lemma:consistency T_P*} to get $N_I$ a cartesian product of models of $T$ satisfying $T_{\Pal}^\ast$. Then it is clear from \cref{lemma:char T_P*} that $M \times N_I \equiv_{\Pal} N_I$.

	(ii) $\Longrightarrow$ (iii) is a direct consequence of \cref{lemma:Palyutin equivalence}.
	
	Assume (iii), and take a filter $\filter$, an $\signature$-structure $M'$ and $N_\filterA$ a reduced product of models of $T$ such that $M^\filter \times M' \simeq N_\filterA$. All $N_i$'s are models of $T_{\Pal}$ and so $M$ is also a model of $T_{\Pal}$ since $T_{\Pal}$ is preserved under reduced products, cartesian factors and reduced roots.
\end{proof}

Now the characterization of Palyutin theories (i.e.\ theory axiomatizable by Palyutin sentences) is a direct consequence of the previous fact :

\begin{corollary}\label{thm:pres Palyutin theory}
$T$ is preserved under reduced products, reduced roots and (finite) cartesian factors if and only if $T$ is axiomatizable by a Palyutin theory.
\end{corollary}

Now a simple application of the Compactness Theorem gives us :

\begin{corollary}\label{thm:pres Palyutin formula}
Let $\phi$ be a formula. $\phi$ is bipreserved under reduced products if and only if it is equivalent to a Palyutin formula.
\end{corollary}

\begin{proof}
	We only prove one implication, the other being already known. Assume that $\phi$ is preserved  under reduced products and reduced factors. We already know that $\phi$ is equivalent to a Palyutin theory $T$. We have $T \models \phi$ so by Compactness Theorem and the stability of Palyutin formulas under $\wedge$, there exists some Palyutin sentence $\psi$ such that $T \models \psi$ and $\psi \models \phi$. Hence, $\psi$ is equivalent to $\phi$.
\end{proof}

Similar techniques can be used to derive preservation theorems for related fragments of first-order logic, like Boolean combinations of Palyutin formulas, positive combinations of Palyutin formulas or Horn combinations of Palyutin formulas. See \ref{sec:preservation theorems} and in particular \ref{table:axiomatisability of continuous theories} for more details on that in the case of continuous logic, which of course can be adapted for classical first-order logic.

\section{Reduced products of metric structures}
\label{sec:reduced products}

Until the end of the paper, $\signature$ will be a fixed \emph{metric} signature.

The definition of the reduced product for metric structures is due to Lopes and more details can be found in \cite{lopesReducedProductsSheaves2013} if necessary. Before giving this defintion, we need to recall some facts about $\limsup$. If $(u_i)_{i \in I}$ is a bounded family of real numbers and $\filter$ is a filter on $I$, we define:
\begin{equation*}
	\limsup_{i \to \filter} u_i \defeq \inf_{J \in \filter} \sup_{i \in J} u_i ~.
\end{equation*}
Note that if $\filter$ is the Fréchet filter on $\omega$, $\limsup_{i \to \filter}$ is the more common operator $\limsup_{i \to \infty}$ and the general case actually shares most properties with this particular one. If $\filter$ is the trivial filter $\{I\}$ then $\limsup_{i \to \filter} u_i$ is simply $\sup_{i \in I} u_i$. We also define
\begin{equation*}
	\liminf_{i \to \filter} u_i \defeq \sup_{J \in \filter} \inf_{i \in J} u_i~.
\end{equation*}
In the case where $\filter$ is an ultrafilter, $\limsup_{i \to \filter}$ and $\liminf_{i \to \filter}$ both coincide with $\lim_{i \to \filter}$. More generally, if $\filter$ is a filter and if $\hat{\filter}$ is the set of all ultrafilters containing $\filter$, then the following hold :
\begin{align*}
	\limsup_{i \to \filter} u_i &= \sup_{\ultrafilter \in \hat{\filter}} \lim_{i \to \ultrafilter} u_i ~;\\
	\liminf_{i \to \filter} u_i &= \inf_{\ultrafilter \in \hat{\filter}} \lim_{i \to \ultrafilter} u_i ~.
\end{align*}
This easily implies that $\limsup$ and $\liminf$ both commute with nondecreasing maps and also the following, for any families of real numbers $(u_i)_{i \in I}$ and $(v_i)_{i \in I}$,
\begin{align}
	\limsup_{i \to \filter} \max(u_i, v_i) = \max \left( \limsup_{i \to \filter} u_i, \limsup_{i \to \filter} v_i \right) \label{eq:prop of limsup}\\
	\liminf_{i \to \filter} (u_i + v_i) \leq \limsup_{i \to \filter} u_i + \liminf_{i \to \filter} v_i \leq \limsup_{i \to \filter} (u_i + v_i)~.
\end{align}

Let us also state another easy property that will be used in \ref{sec:palyutin formulas}.
\begin{lemma}\label{lemma:limsup and quantifiers}
Let $(X_i)_{i \in I}$ be a family of sets and $(u_i)_i \in \prod_i \real^{X_i}$ be a uniformly bounded family of real valued maps, and $\filter$ a filter on $I$. Then
\begin{align*}
	\limsup_{i \to \filter} \inf_{x \in X_i} u_i(x) & = \inf_{x \in \prod_i X_i} \limsup_{i \to \filter} u_i(x_i) \\
	\limsup_{i \to \filter} \sup_{x \in X_i} u_i(x) & = \sup_{x \in \prod_i X_i} \limsup_{i \to \filter} u_i(x_i) ~.\\
\end{align*}
\end{lemma} 

Now fix a metric signature $\signature$. Let $M = (M_i)_{i \in I}$ be a family of $\signature$-structures and $\filter$ a filter on $I$. We define a pseudo-distance $d_{\filter}$ on $\prod_{i \in I} M_i$ as :
\begin{equation*}
	d_{\filter}(a,b) \defeq \limsup_{i \to \filter} d^{M_i}(a_i, b_i)~.
\end{equation*}
We define a relation $\sim$ on $\prod_{i \in I} M_i$ where $a \sim b$ if $d_{\filter}(a,b) = 0$. Then $d_{\filter}$ defines a distance also denoted by $d_{\filter}$ on $D \defeq \prod_{i \in I} M_i / \sim$. We can check that $(D, d_{\filter})$ is a complete bounded metric space, that will be the domain of the reduced product. Let $\pi$ be the canonical quotient map (with dense image) $\prod_{i} M_i \to D$. We now define an $\signature$-structure $M_{\filter}$ with domain $D$ where :
\begin{itemize}
	\item For every constant symbol $c$ in $\signature$, $c^{M_{\filter}} \defeq \pi \left( (c^{M_i})_{i \in I} \right)$;
	\item For every function symbol $f$ in $\signature$, and every $a \in \prod_i M_i$ of the good sort,
	\[f^{M_{\filter}} \left( \pi(a) \right) \defeq \pi \left( (f^{M_i}(a_i))_{i \in I} \right) ~;\]
	\item For every predicate symbol $p$ in $\signature$, and every $a \in \prod_i M_i$ of the good sort,
	\[p^{M_{\filter}} \left( \pi(a) \right) \defeq \limsup_{i \to \filter} p^{M_{\filter}}(a_i) ~.\]
\end{itemize}
Of course, one must check that these definitions make sense and do indeed define an $\signature$-structure (recall that $\signature$ contains information uniformizing smoothness and bounds of the interpretations). This $\signature$-structure will be denoted by $\prod_i M_i / \filter$ or more conveniently by $M_{\filter}$ and is called the \emph{reduced product of $M$ by $\filter$}. If $a = (a_i)_{i \in I} \in \prod_i M_i$, then we denote by $a_\filter$ the element $\pi(a) \in M_\filter$. If all $M_{\filter}$'s are equal to some $\signature$-structure $N$, then $M_{\filter}$ is also denoted by $M^{\filter}$ and is called the \emph{reduced power of $M$ by $\filter$}. In the case where $\filter$ is an ultrafilter, the reduced product $M_{\filter}$ is in fact the ultraproduct usually denoted by $\prod_i M_i / \filter$. 

\begin{remark}
	When identifying a classical structure with a metric one thanks to the discrete metric, it is easy to check that this definition of the reduced product coincides with the classical one introduced by Łoś.
\end{remark}

The following gives a precise meaning of what we will call preservation and copreservation in the rest of the paper.

\begin{definition}
	Let $\phi(x)$ be a continuous formula. We say that:
	\begin{itemize}
		\item $\phi$ is \emph{preserved under reduced products} if for every family of structures $M = (M_i)_{i \in I}$ and every filter $\filter$ on $I$, $\phi^{M_{\filter}}(a_{\filter}) \leq \limsup_{i \to \filter} \phi^{M_i}(a_i)$;
		\item $\phi$ is \emph{copreserved under reduced products} if for every family of structures $M = (M_i)_{i \in I}$ and every filter $\filter$ on $I$, $\limsup_{i \to \filter} \phi^{M_i}(a_i) \leq \phi^{M_{\filter}}(a_{\filter})$;
		\item $\phi$ is \emph{bipreserved under reduced products} if it is both preserved and copreserved under reduced products.
	\end{itemize}
\end{definition}

A straightforward induction on terms proves that any atomic formula is bipreserved under reduced products. Of course, this is not the case at all when we replace "atomic formula" by "continuous formula".

\begin{notation}
	If $M = (M_i)_{i \in I}$ is a family of $\signature$-structures and $T$ is a theory, then we denote the set $\{i \in I : M_i \models T\}$ by $\indices{T}_M$, or simply by $\indices{T}$ when the context makes it clear. If $T$ consists of a single condition of the form $\phi \leq r$, we will also write $\indices{T} = \indices{\phi \leq r}_M = \indices{\phi \leq r}$. Abusing of notations, we will sometimes also write $\indices{\phi < r} = \indices{\phi < r}_M \defeq \{i \in I : \phi^{M_i} < r\}$ even though $\{\phi < r\}$ is not a continuous theory \emph{per se}.
	
	If $M = (M_i)_{i \in I}$ is a family of $\signature$-structures, $a \in \prod_i M_i^x$ and $p(x)$ is a partial type, then we will also write $\indices{p(a)}$ to mean $\indices{p(x)}_{(M_i,a_i)_i}$ i.e.\ the set $\{i \in I : a_i \models p\}$. Again, we will abuse of this notation and use it also replacing $p$ by single conditions of various forms.
\end{notation}

\begin{definition}
	Let $T$ be a continuous theory.
	\begin{itemize}
		\item We say that $T$ is \emph{preserved under reduced products} if for every family of structures $M = (M_i)_{i \in I}$ and every filter $\filter$ on $I$, if the set $\indices{T} \in \filter$, then $M_{\filter} \models T$.
		\item We say that $T$ is \emph{preserved under reduced roots} if for every structure $M$ and every filter $\filter$ on $I$, if $M^\filter \models T$, then $M \models T$.
		\item We say that $T$ is \emph{preserved under (finite) cartesian factors} if for every (finite) family of $\signature$-structures $(M_i)_{i \in I}$ such that $\prod_{i \in I} M_i \models T$, then $M_i \models T$ for all $i \in I$
	\end{itemize}
\end{definition}

Let us make explicit the link between preservation of sentences and of theories.

\begin{proposition}
	Let $\phi$ be a continuous sentence. The following are equivalent:
	\begin{enumerate}[(i)]
		\item The sentence $\phi$ is preserved under reduced products;
		\item For every $r \in \real$, the theory $\{\phi \leq r\}$ is preserved under reduced products.
	\end{enumerate}
\end{proposition}

\begin{proof}
	Suppose (i). Let $r \in \real$, $M = (M_i)_{i \in I}$ a family of structures and $\filter$ a filter on $I$ and assume that $\indices{\phi \leq r} \in \filter$. Hence, we have
	\[\phi^{M_\filter} \leq \limsup_{i \to \filter} \phi^{M_i} = \inf_{J \in \filter} \sup_{i \in J} \phi^{M_i}
		\leq \sup_{i \in \indices{\phi \leq r}} \phi^{M_i} \leq r ~,\]
	which exactly means that $M_\filter \models \phi \leq r$, thus (ii) holds.
	
	Now, suppose (ii).  Let $M = (M_i)_{i \in I}$ be a family of structures and $\filter$ a filter on $I$. Put $r \defeq \limsup_{i \to \filter} \phi^{M_i}$ and let $\epsilon > 0$. By definition of $\limsup$, there exists a $J \in \filter$ such that $\sup_{i \in J} \phi^{M_i} \leq r + \epsilon$, i.e.\ $J \subseteq \indices{\phi \leq r + \epsilon}$. Since
\end{proof}

A similar characterization exists for copreservation of sentences but it involves conditions of the form $\phi < r$ which are not axiomatisable in continuous logic \emph{a priori}.

\begin{proposition}
	Let $\phi$ be a continuous sentence. The following are equivalent:
	\begin{enumerate}[(i)]
		\item $\phi$ is copreserved under reduced products;
		\item For every $r \in \real$, every family of structures $M = (M_i)_{i \in I}$ and every filter $\filter$ on $I$, if $\phi^{M_{\filter}} < r$, then $\indices{\phi < r} \in \filter$.
	\end{enumerate}
\end{proposition}

\begin{proof}
	Suppose (i). Let $r \in \real$, $M = (M_i)_{i \in I}$ a family of structures and $\filter$ a filter on $I$ and assume that $\phi^{M_{\filter}} < r$. Hence, we have
	\[\inf_{J \in \filter} \sup_{i \in J} \phi^{M_i} = \limsup_{i \to \filter} \phi^{M_i} \leq \phi^{M_\filter} < r~,\]
	so there exists $J \in \filter$ such that $\sup_{i \in J} \phi^{M_i} < r$ which exactly means that $J \subseteq \indices{\phi < r}$. Since $\filter$ is a filter, $\indices{\phi < r} \in \filter$ thus (ii) holds.
	
	Now, suppose (ii).  Let $M = (M_i)_{i \in I}$ be a family of structures and $\filter$ a filter on $I$. Take any $r > \phi^{M_\filter}$, then $\indices{\phi < r} \in \filter$ which means that 
	\[\limsup_{i \to \filter} \phi^{M_i} = \inf_{J \in \filter} \sup_{i \in J} \phi^{M_i} \leq \sup_{i \in \indices{\phi < r}} \phi^{M_i} \leq r ~.\]
	Thus, $\limsup_{i \to \filter} \phi^{M_i} \leq \phi^{M_\filter}$.
\end{proof}

\subsection*{Continuous Horn formulas}

Goldbring and Keisler introduced an analogue of Horn formulas in the framework of continuous logic \cite{goldbringContinuousSentencesPreserved2022}. They used the name \emph{conditional formula} but we prefer to push the analogy further by talking about \emph{Horn formulas}. We review this notion here, both for the sake of completeness and to state \ref{thm:characterization of Horn sentences}, which is not new in nature but in its presentation.

\begin{definition}[\cite{goldbringContinuousSentencesPreserved2022}, Definition 3.5]
	A formula $\phi$ is \emph{primitive Horn} if there are atomic formulas $\alpha,\beta_1, \dots, \beta_n$ a unary nondecreasing connnective $C$ and unary nonincreasing connectives $D_1,\dots,D_n$ such that
	\[\phi = \min(C(\alpha),D_1(\beta_1),\dots,D_n(\beta_n)) ~.\]
	The set of \emph{Horn formulas} is the smallest set containing primitive Horn formulas and closed under the $\max$ connective and the quantifiers $\sup$ and $\inf$. A Horn formula with no free variables is a \emph{Horn sentence}. A \emph{Horn condition} is a condition of the form $\phi \leq r$ where $\phi$ is a Horn sentence and $r \in \real$.
\end{definition}

\begin{remark}\label{rk:Horn formulas are closed under nondecreasing connectives}
	Using a straightforward induction on the defintion of Horn formulas, one can prove that if $\phi$ is a Horn formula and $C$ is a nondecreasing unary connective then $C \phi$ is equivalent to a Horn formula.
\end{remark}

The reason why we introduced Horn formulas is because they correspond essentially to the first-order formulas that are preserved under reduced products. It is easy to prove by induction that Horn formulas are preserved under reduced products.

Let $T$ be a continuous $\signature$-theory. Let $T_\Horn \defeq \{ \phi \leq r : \phi \text{ Horn condition}\}$. Goldbring and Keisler essentially proved the following.

\begin{fact}\label{fact:characterization of T_H}
	Let $M$ be an $\signature$-structure. The following are equivalent :
	\begin{enumerate}[(i)]
		\item $M \models T_{\Horn}$;
		\item Some ultrapower of $M$ is isomorphic to some reduced product of models of $T$.
	\end{enumerate}
	As a consequence, if $T$ is preserved under reduced products then $T$ can be axiomatised by Horn conditions.
\end{fact}

The original proof of Goldbring and Keisler uses CH and then relies on absoluteness argument to show that it is a theorem in ZFC. A more direct proof can be found in the author's Master Thesis \cite{fronteauProduitsReduitsLogique2023}, in the proof of Théorème 4.14. An application of the Compactness Theorem leads to the following.

\begin{corollary}\label{thm:characterization of Horn sentences}
	Let $\phi(x)$ be a continuous formula. The following are equivalent :
	\begin{enumerate}[(i)]
		\item $\phi$ is preserved under reduced products ;
		\item $\phi$ is approximable by Horn formulas.
	\end{enumerate}
\end{corollary}

We take the opportunity of this corollary to introduce a technical lemma that will be used in similar applications later.

\begin{lemma}\label{lemma:approximation by B-combination}
	Let $\phi(x)$ be a continuous formula. Let $\epsilon > 0$ and $r_0 < \dots < r_k$ such that $r_{i + 1} - r_i = \epsilon$ for every $i < k$ and $r_0 \leq \phi \leq r_k$ in every $\signature$-structure. Let $\psi_0, \dots \psi_{k - 1}$ be continuous formulas. Assume that we can find $\lambda_i$'s such that for every $i < k$, $\lambda_i > 0$ and :
	\begin{align*}
			\phi \leq r_i & \models \psi_i \leq 0 \\
			\psi_i \leq \lambda_i & \models \phi < r_{i +  1} ~.
	\end{align*}
Then there are nondecreasing unary connectives $C_0,\dots,C_{k - 1}$ such that the formula defined by $\theta \defeq \max(C_0 \psi_0, \dots, C_{k -1} \psi_{k-1})$ satisfies $\|\phi - \theta\| \leq 2 \epsilon$.
\end{lemma}

\begin{proof}
	Define $\theta_i \defeq \max\left( r_0, \min \left( r_{i + 1}, \frac{r_{i + 1} - r_0}{\lambda_i} \psi_i  + r_0 \right) \right)$ and $\theta \defeq \max(\theta_0, \dots , \theta_{k - 1})$. It is easy to see that 
	\[\begin{cases} & \models r_0 \leq \theta_i \leq r_{i + 1} \\
		\phi \leq r_i & \models \theta_i = r_0 \\
		\phi \geq r_{i + 1} & \models \theta_i = r_{i + 1} \text{ for every $0 \leq i < k$} ~.
		\end{cases} \]
	Thus, for every $0 \leq i < k$,
	\begin{align*}
	r_i \leq \phi \leq r_{i + 1} & \models \theta = \max(r_1, \dots, r_i, \theta_i, r_0, \dots, r_0)\\
	& \models r_i \leq \theta \leq r_{i + 1} ~,
	\end{align*}	
	proving that $\|\theta - \phi \| \leq 2 \epsilon$.
\end{proof}

\begin{proof}[Proof of \ref{thm:characterization of Horn sentences}]
	As explained above, the proof of (ii) $\Rightarrow$ (i) relies on a simple induction.
	
	Now, suppose that $\phi$ is preserved under reduced products. Let $\epsilon > 0$ and $r_0 < \dots < r_k$ be as in the assumptions of \ref{lemma:approximation by B-combination}. For every $0 \leq i < k$, the theory $\{\phi \leq r_i\}$ is preserved under reduced products so it can be axiomatised by a set of Horn conditions $T_i$. By Compactness and since Horn formulas are closed under the $\max$ connective and essentially by nondecreasing unary connectives by \ref{rk:Horn formulas are closed under nondecreasing connectives}, one can find a Horn formula $\psi_i$ and $\lambda_i > 0$ such that
	\[\phi \leq r_i \models \psi_i \leq 0 \text{ and } \psi_i \leq \lambda_i \models \phi \leq r_i + \epsilon / 2 ~.\]
	Hence, by \ref{lemma:approximation by B-combination} and using once again the closure of Horn formulas under $\max$ and nondecreasing connectives, we can find Horn formulas arbitrarily close to $\phi$.
\end{proof}

\section{Continuous Palyutin formulas}
\label{sec:palyutin formulas}

In this section we generalise results from \cite{palyutinCategoricalHornClasses1980} to continuous logic. In particular, we give an analogue of what is already known as \emph{Palyutin formulas} or \emph{h-formulas} in classical model theory.

\begin{definition}\label{def:Palyutin formulas}
The set of \emph{Palyutin formulas} is the smallest set such that :
\begin{itemize}
	\item Every atomic formula is a Palyutin formula;
	\item If $\phi$ is a Palyutin formula and $C$ is a nondecreasing unary connective, then $C \phi$ is a Palyutin formula;
	\item If $\phi, \psi$ are Palyutin formulas, $\max(\phi, \psi)$ is a Palyutin formula;
	\item If $\phi$ is a Palyutin formula and $x$ is a variable, $\inf_x \phi$, and $\sup_x \phi$ are Palyutin formulas;
	\item If $\phi, \psi$ are Palyutin formulas, $D$ is a nonincreasing unary connective with fixed point $\Delta$ and $x$ is a variable, $\max \left(\inf_x \phi, \sup_x \min(D \phi, \Delta, \psi) \right)$ is a Palyutin formula.
\end{itemize}
A \emph{Palyutin sentence} is a Palyutin formula with no free variables, and a Palyutin condition is a condition of the form $\phi \leq r$ where $\phi$ is a Palyutin sentence and $r \in \real$.
\end{definition}

\begin{remark}Let $D$ be a nonincreasing unary connective. Then, the map $t \mapsto t - D t$ is continuous, its limits at $- \infty$ is $- \infty$ and at $+ \infty$ is $+ \infty$. Hence, there exists a $\Delta \in \real$ such that $\Delta = D \Delta$. Moreover, this fixed point is unique since if $\Delta_1 \leq \Delta_2$ with $D \Delta_i = \Delta_i$, then 
\[\Delta_1 \leq \Delta_2 = D \Delta_2 \leq D \Delta_1 ~,\]
which implies that $\Delta_1 = \Delta_2$. Hence, the last item of \ref{def:Palyutin formulas} covers every nonincreasing unary connective, and there is no choice for the constant $\Delta$.
\end{remark}

As Horn formulas, Palyutin formulas are preserved under reduced products, but they are also copreserved. As usual, proving that our defined Palyutin formulas have the expected preservation properties is easy, but we write details here in particular for the last case to make sense of the appearance of fixed points in the definition we suggest.

\begin{proposition}\label{thm:preservation of Palyutin formulas}
	Palyutin formulas are bipreserved under reduced products.
\end{proposition}

\begin{proof}
	The proof is a straightforward induction.
	\begin{description}
		\item[Atomic formulas] Easy.
		\item[Nondecreasing unary connective] Use the fact that a nondecreasing map commutes with the $\limsup$.
		\item[Maximum] Use \ref{eq:prop of limsup}.
		\item[Quantifiers] Use \ref{lemma:limsup and quantifiers}.
		\item[The last case] It only remains to prove that if $\phi(x,y), \psi(x,y)$ are bipreserved under reduced products, $D$ is a nonincreasing unary connective with fixed point $\Delta$ and $x$ is a variable, then $\theta \defeq \max(\inf_y \phi, \sup_y \min(D\phi, \Delta, \psi))$ is also bipreserved under reduced products. Let us denote $\inf_y \phi$ by $\theta_1$ and $\min(D\phi, \Delta, \psi)$ by $\theta_2$. Consider a reduced product $M_\filter = \prod_i M_i / \filter$ and $a \in \prod_i M_i^x$. Suppose first that $\indices{\theta(a) \leq r} \in \filter$ for some $r \in \real$ and let us show that $\theta^{M_\filter}(a_\filter) \leq r$. By \ref{lemma:limsup and quantifiers}, we already know that $\theta_1^{M_\filter}(a_\filter) \leq r$. Let $b \in \prod_i M_i^y$.
		\begin{itemize}
			\item If $\Delta \leq r$ or $D \phi^{M_\filter}(a_\filter,b_\filter) \leq r$, it is clear that $\theta_2^{M_\filter}(a_\filter,b_\filter \leq r$.
			\item Otherwise, then since $\phi$ is copreserved, $\indices{D \phi(a,b) > r}_{(M_i,a_i,b_i)_i} \in \filter$. But $\indices{\theta(a) \leq r} \in \filter$ so $\filter{\theta_2(a,b) \leq r} \in \filter$. By taking the intersection, we deduce that $\indices{\psi(a,b) \leq r} \in \filter$, so by preservation of $\psi$, we have $\theta_2^{M_\filter}(a_\filter, b_\filter) \leq r$.
		\end{itemize}
	Now suppose that $\Delta < r$ for some $r \in \real$, and let us show that $\indices{\theta(a) < r} \in \filter$. First, $\indices{\theta_1(a) < r} \in \filter$ by \ref{lemma:limsup and quantifiers}. It then suffices to show that $\indices{\sup_y \theta_2(a,y) < r} \in \filter$ and we could conclude by taking an intersection. In the case where $\Delta < r$, it is clear, so we now assume that $r \leq \Delta$. We build $b \in \prod_i M_i$ such that for every $i \in I$:
	\begin{enumerate}[(1)]
		\item $\min(D \phi^{M_i}(a_i,b_i), \psi(a_i, b_i)) \geq r$ if such a $b_i$ exists;
		\item $\phi^{M_i}(a_i) < r$ else and if such a $b_i$ exists;
		\item $b_i$ is arbitrary else.
	\end{enumerate}
	Let us fix $i \in I$. If $\theta_1(a_i) < r$, then:
	\begin{itemize}
		\item Either we are in case (1) and then in particular $D \phi^{M_i}(a_i,b_i) \geq r$;
		\item Either we are in case (2) and then $D \phi^{M_i}(a_i,b_i) \geq D r \geq D(\Delta) = \Delta \geq r$.
	\end{itemize}
	We thus deduce that $\indices{D\phi \geq r} \supseteq \indices{\theta < r} \in \filter$. By preservation of $\phi$, we thus have $D \phi^{M_\filter}(a_\filter,b_\filter) \geq r$ and since $\theta_2^{M_\filter}(a_\filter, b_\filter) < r$, then necessarily $\psi^{M_\filter}(a_\filter, b_\filter) < r$. By construction of $b$ and copresevation of $\psi$, we get
	\[\indices{\sup_y \theta_2(a,y) < r} \supseteq \indices{\psi(a,b) < r} \in \filter ~. \qedhere\]
	\end{description}
\end{proof}

Of course, this allows us to state the following preservation results for theories axiomatised by Palyutin conditions.

\begin{corollary}\label{thm:pres Palyutin theories}
	Let $T$ be a theory axiomatised by Palyutin conditions. Then $T$ is preserved under reduced products, reduced roots and cartesian factors.
\end{corollary}

\begin{proof}
	Let $\phi \leq r$ be a Palyutin condition in $T$. If $M = (M_i)_{i \in I}$ is a family of structures and $\filter$ is a filter such that $\indices{T} \in \filter$, then $\indices{T} \subseteq \indices{\phi \leq r} \in \filter$ and thus $\phi^{M_{\filter}} \leq r$ by preservation of $\phi$. If $M$ is a structure and $\filter$ is a filter such that $M^{\filter} \models T$, then $\phi^M = \phi^{M^{\filter}} \leq r$. Finally, if $M = (M_i)_{i \in I}$ is a family of structures such that $\prod_i M_i \models T$, then for every $i \in I$, $\phi^{M_i} \leq \sup_j \phi^{M_j} = \phi^{\prod_j M_j} \leq r$.
\end{proof}

\begin{definition}
	Let $\SCP$ be the theory containing
	\[\sup_{y} \left( \inf_x \max\left(\phi, D_1 \psi_1, \dots, D_n \psi_n \right) - \max_{j = 1}^n \inf_x \max \left(\phi, D_j \psi_j \right) \right) \leq 0 ~,\]
for every $n \geq 2$, Palyutin formulas $\phi(x, y), \psi_1(x, y), \dots, \psi_n(x, y)$ and $D_1,\dots,D_n$ nondecreasing unary connectives.
\end{definition}

\begin{lemma}\label{lemma:pres SCP cont}
	The theory $\SCP$ is preserved under reduced products.
\end{lemma}

\begin{proof}
	Let $\filter$ be a filter on $I$, and $(M_i)_{i \in I}$ a family of $\signature$-structures with $M_i \models \SCP$ for every $i \in I$. Let $\phi(x,y), \psi_1(x, y),\dots, \psi_n(x, y)$ be Palyutin formulas, $D_1,\dots,D_j$ nondecreasing unary connectives and $b \in \prod_i M_i^y$. Define
	\[r \defeq \max_{j = 1}^n \inf_x \max \left( \phi^{M_\filter}(x,b_\filter), D_j \psi^{M_\filter}_j(x,b_\filter) \right) ~.\]
	Note that $\inf_x \phi^{M_\filter}(x,b_\filter) \leq r$ so $\indices{\inf_x \phi(x,b) < r + \epsilon} \in \filter$ for every $\epsilon > 0$ by co-preservation of $\phi$. Fix $\epsilon > 0$ and define for $i \in I$
	\[J(i) \defeq \left\{ j \in \{1,\dots,n\} : \inf_x \max \left( \phi^{M_i}(x,b_i), D_j \psi^{M_i}_j(x,b_i) \right) < r + \epsilon\right\} ~.\]
	We can now find $a$ such that $\max(\phi^{M_i}(a_i,b_i), \max_{j \in J(i)} D_j \psi_j^{M_i}(a_i, b_i)) < r + \epsilon$\footnote{By convention, $\max_{\emptyset} = - \infty$.} for every $i \in \indices{\inf_x \phi(x,b) < r + \epsilon}$. By preservation of $\phi$, $\phi^{M_\filter}(a_\filter, b_\filter) \leq r + \epsilon$. Now fix $k \in \{1,\dots,n\}$. We claim that $X_k \defeq \indices{\inf_x \max( \phi(x,b), D_k \psi_k(x,b)) \geq r + \epsilon} \not \in \filter$. If not, then for every $c \in \prod_i M_i$ such that $M_\filter \models \max( \phi(c_\filter, b_\filter), D_k \psi_k(c_\filter, b_\filter)) < r + \epsilon$, we would have
	\[\llbracket D_k \psi_k(c, b) \geq r + \epsilon \rrbracket \supseteq \llbracket \varphi(x,b) < r + \epsilon \rrbracket \cap X_k \in \filter ~,\]
	so by preservation of $\psi_k$, $M_\filter \models D_k \psi_k(c_\filter, b_\filter) \geq r + \epsilon$. Thus, we would have
	\[M_\filter \models \inf_x \max \left( \phi^{M_\filter}(x, b_\filter), D_k \psi^{M_\filter}_k(x, b_\filter) \right) \geq r + \epsilon\]
	which contradicts our assumption, hence $X_k \not \in \filter$. By construction, we have
	\[k \not\in J(i) \text{ for every $i \in \indices{D_k \psi_k(a,b) \geq r + \epsilon} \cap \indices{\inf_x \phi(x, b) < r + \epsilon}$,}\]
	hence
	\[\indices{D_k \psi_k(a, b) \geq r + \epsilon} \cap \indices{\inf_x \phi(x, b) < r + \epsilon} \subseteq X_k \not \in \filter ~.\]
	 We deduce that $\indices{D_k \psi_k(a, b) > r + \epsilon} \not \in \filter$. By copreservation of $\psi_k$, we know that $D_k \psi_k^{M_\filter}(a_\filter, b_\filter) \leq r + \epsilon$. We thus have :
	\begin{equation*}
	M_\filter \models \inf_x \max \left( \phi(x, b_\filter), \max_{j = 1}^n D_j \psi_j(x, b_\filter \right) \leq r ~.\qedhere
	\end{equation*}
\end{proof}

Recall that the Fréchet filter on $\omega$ is the one consisting of all cofinite subsets of $\omega$.

\begin{lemma}\label{lemma:atomless boolean algebra SCP cont}
	Let $(M_i)_{i \in I}$ be a family of $\signature$-structures and $\filter$ a filter. If $\{0,1\}^{\filter}$ is an atomless boolean algebra, then $M_\filter \models \SCP$. In particular, it holds when $\filter$ is the Fréchet filter on $\omega$.
\end{lemma}

\begin{proof}
We will use the following fact.
	\begin{fact}\label{lemma:atomless subsets}
		If $\{0,1\}^\filter$ is atomless and $W_1,\dots,W_n \in 2^I$ are such that $W_1,\dots,W_n > 0$ in $\mathbf \{0,1\}^\filter$, then we can find $X_i \subseteq W_i$ such that for all $i \in \{1 ,\dots, n\}$:
		\begin{itemize}
			\item $X_i > 0$ in $\{0,1\}^\filter$;
			\item $X_i \cap X_j = \emptyset$ if $i \neq j$.
		\end{itemize}
	\end{fact}
	
	Assume that $\{0,1\}^\filter$ is atomless. Let $\phi(x,y), \psi_1(x,y), \dots, \psi_n(x, y)$ be Palyutin formulas, and $D_1,\dots, D_k$ be nondecreasing unary connectives. We define
	\[r \defeq \max_{j = 1}^n \inf_x \max \left( \phi^{M_\filter}(x,b), D_j \psi^{M_\filter}_j(x,b) \right) ~.\]
	
	Let $\epsilon > 0$, $U \defeq \indices{\inf_x \phi(x,b) < r + \epsilon}$, and :
	\[W_k \defeq \indices{\inf_x \max (\varphi(x,b), D_k \psi_k(x,b)) < r + \epsilon}\]
	Necessarily, $I \setminus W_k = \indices{\inf_x \max (\varphi(x,b), D_k \psi_k(x,b)) \geq r + \epsilon} \not\in \filter$ as in the proof of \cref{lemma:pres SCP cont}. It is also clear that $U \in \filter$ since $\inf_x \phi(x,b)$ is copreserved. We can now apply \cref{lemma:atomless subsets} to find $X_k \subseteq W_k$ pairwise disjoint such that $I \setminus X_k \not\in \filter$ for every $k$. We can now define $a \in M_i^x$ such that:
	\begin{itemize}
		\item $\max \left( \varphi^{M_i}(a_i, b) , D_k \psi_k^{M_i}(a_i, b) \right) < r + \epsilon$ if $i \in X_k$;
		\item $\phi^{M_i}(a_i, b) < r + \epsilon$ if $i \in U \setminus (X_1 \cup \dots \cup X_k)$.
	\end{itemize}
	It is easy to see that $\indices{\varphi(a,b) \leq r + \epsilon} \supseteq U \in \filter$ and $\indices{D_k \psi_k(a,b) > r + \epsilon} \subseteq I \setminus X_k \not\in \filter$ for every $k$ so by preservation of $\phi$ and copreservation of $\psi_k$ we get
	\[M_\filter \models \max \left(\varphi(a, b), \max_{j = 1}^n D_j \psi_j(a, b) \right) \leq r + \epsilon ~,\]
	we thus have
		\begin{equation*}
		M_\filter \models \inf_x \max \left(\varphi(x,b), \max_{j = 1}^n D_j \psi_j(x, b) \right) \leq r ~.\qedhere
		\end{equation*}
\end{proof}

Let $\mathcal B := \{C : \real \to \real : n < \omega, C \text{ affine}\} \cup \{\max, \min\}$. We call $\mathcal B$-combination of Palyutin formulas any formula which can be built using only Palyutin formulas, and connectives in $\mathcal B$. We will prove that in the theory $\SCP$, $\mathcal B$-combination of Palyutin formulas are dense in the set of continuous formulas. First, we prove the following lemma which is some sort of quantifier elimination for $\mathcal B$-combination of Palyutin formulas.

\begin{lemma}\label{lemma:approx EB-combination}
	If $\phi(x,y)$ is a $\mathcal B$-combination of Palyutin formulas and $\epsilon > 0$, then there exists a $\mathcal B$-combination of Palyutin formulas $\psi(x)$ such that :
	\[\SCP \models \sup_{x} |\inf_y \phi - \psi| \leq \epsilon\]
\end{lemma}

\begin{proof}[Proof of \ref{lemma:approx EB-combination}]
	First, we may assume that $\phi$ is of the form :
	$$\phi(x,y) = \min_{i = 1}^n \max \left( \theta_i, D_{i, 1} \gamma_{i,1}, \dots, D_{i, l_i} \gamma_{i, l_i} \right)$$
	where $\theta_i$'s, $\gamma_{i,j}$'s are Palyutin formulas and $D_{i,j}$'s are nondecreasing unary connectives. It is then clear that :
	\[\SCP \models \sup_{\bar x} |\inf_y \phi - \min_{i = 1}^n \max_{j = 1}^{l_i} \inf_{y} \max (\theta_i, D_{i, j} \gamma_{i,j}) | = 0 \]
	Thus, it suffices to show that if $\theta, \gamma$ are Palyutin formulas and $D$ is a nondecreasing unary connective, then there exists some $\mathcal B$-combination of Palyutin formulas $\psi$ such that :
	\begin{equation*}\label{eq:approx inf theta D gamma}
	\|\inf_{y} \max \left( \theta, D \gamma \right) - \psi \| \leq \epsilon
	\end{equation*} 
	Take $r_0 < r_1 < \dots < r_k$ such that $\inf_{y} \max \left( \theta, D \gamma \right)$ is bounded by $r_0$ and $r_k$ and $|r_{i + 1} - r_i| = \epsilon$ for every $i$. Define $\rho_i := r_i + \epsilon / 2$, $d_i : t \mapsto 2 \rho_i - t$ and $\lambda := \epsilon / 6$. Note that the fixed point of $d_i$ is $\rho_i$ and $\rho_i = r_i + 3 \lambda$. Finally, define :
	\begin{align*}
	\psi_i \defeq & \max \left( \inf_y \theta - r_i,
	d_i \left( \max \left( \inf_y \theta, \sup_y \min \left( d_i \theta, d_i D \gamma, \rho_i \right) \right) \right) - \rho_i
	\right) \\
	= & \max \left( \inf_y \theta - r_i,
	\min \left( \rho_i - \inf_y \theta, \inf_y \max \left( \theta, D \gamma, \rho_i \right) - \rho_i \right)
	\right)
	\end{align*}
	$\psi_i$ is a $\mathcal B$-combination of Palyutin formulas. For all $i$, we have on one hand
	\[\inf_{y} \max \left( \theta, D \gamma \right) \leq r_i \models \psi_i \leq 0 ~,\]
	and on the other hand
	\[\psi_i \leq \lambda \models \left(\inf_{y} \max \left( \theta, D \gamma \right) \leq r_i + 2 \epsilon / 3 \right) ~.\]
	Now we conclude using \ref{lemma:approximation by B-combination}.
\end{proof}

\begin{lemma}\label{lemma:SCP B-combination cont}
	For every first-order formula $\phi(x)$ and $\epsilon > 0$, there exists a $\mathcal B$-combination of Palyutin formulas $\psi(x)$ such that $\SCP \models \sup_{x} \left| \phi - \psi \right| \leq \epsilon$.
\end{lemma}

\begin{proof}
	By the lattice version of Stone-Weierstrass's Theorem, every first-order formula is approximable by $\mathcal B$-formulas i.e.\ formulas built only with connectives in $\mathcal B$ and quantifiers (see Section 6 of \cite{benyaacovModelTheoryMetric2008} for more details on density of restricted formulas). We might thus assume that $\phi$ is a $\mathcal B$-formula. The proof is now an easy induction on the construction of $\mathcal B$-formulas where we use \cref{lemma:approx EB-combination} for the case of quantifiers.
\end{proof}

\begin{corollary}\label{lemma:SCP completeness cont}
	Suppose that $T$ is a Palyutin-complete theory i.e. that for every Palyutin sentence $\phi$, $T \models \phi = r$ for some $r \in \real$. Then $T \cup \SCP$ is a complete theory.
\end{corollary}

\begin{proof}
	Take any first-order formula $\phi$, $M, N \models T \cup \SCP$. From \cref{lemma:SCP B-combination cont}, we get a $\mathcal B$-combination of formulas $\psi$ such that $\SCP \models |\phi - \psi| \leq \epsilon$. Clearly, $\psi^M = \psi^N$ since $T$ is Palyutin-complete, but then :
	$$|\phi^M - \phi^N| \leq |\phi^M - \psi^M| + |\psi^M - \psi^N| + |\psi^N - \phi^N| \leq 2 \epsilon$$
	Thus, $\phi^M = \phi^N$.
\end{proof}

We now write $M \equiv_{\Pal} N$ when $\phi^M = \phi^N$ for every Palyutin sentence $\phi$. Recall that $\frechet$ denotes the Fréchet filter on $\omega$. The following can be seen as an analogue of Keisler-Shelah's Theorem for Palyutin theories, and its proof is identical to the other one established in \ref{lemma:Palyutin equivalence}.

\begin{theorem}\label{lemma:Palyutin equivalence cont}
	Let $M$ and $N$ be two $\signature$-structures. The following are equivalent :
	\begin{enumerate}[(i)]
		\item $M \equiv_{\Pal} N$;
		\item $M^{\frechet} \equiv N^{\frechet}$;
		\item $M^\filter \simeq N^\filter$ for some filter $\filter$.
	\end{enumerate}
\end{theorem}

\begin{proof}
	Assume (i). Let $T$ be the theory containing every $\phi = \phi^M$ where $\phi$ is a Palyutin sentence. Moreover, $M^{\frechet}, N^{\frechet} \models T \cup \SCP$ according to \cref{lemma:atomless boolean algebra SCP cont} so using \cref{lemma:SCP completeness cont}, $M^{\frechet} \equiv N^{\frechet}$.
	
	(ii) $\Longrightarrow$ (iii) is a consequence of Keisler-Shelah's Theorem and (iii) $\Longrightarrow$ (i) is simply due to the bipreservation of Palyutin sentences.
\end{proof}

\begin{theorem}
	Let $T$ be a complete theory. The following are equivalent :
	\begin{enumerate}[(i)]
		\item $T \models \SCP$.
		\item $T$ is preserved under reduced products.
		\item $T$ is preserved under $\frechet$-products meaning that any reduced product of models of $T$ by $\frechet$ is also a model of $T$.
	\end{enumerate}
\end{theorem}

\begin{proof}
	Assume (i). Let $T'$ be the theory consisting of every $\phi = r$ where $\phi$ is a Palyutin sentence, $r \in \real$ and $T \models \phi = r$. $T'$ is Palyutin-complete because $T$ is complete so $T' \cup \SCP$ is complete. Hence, $T' \cup \SCP \models T$. By preservation of $T'$ and $\SCP$, we deduce (ii).
	
	(ii) $\Longrightarrow$ (iii) is obvious and (iii) $\Longrightarrow$ (i) is a consequence of \cref{lemma:atomless boolean algebra SCP cont}.
\end{proof}

\begin{corollary}
	A structure $M$ is a model of $\SCP$ if and only if $M \equiv M^{\frechet}$. In this case, $M$ is elementarily equivalent to all its reduced powers.
\end{corollary}

The previous results give us now obvious examples of models of $\SCP$ : in the classical setting, every atomless boolean algebra, infinite vector space, infinite sets are all models of $\SCP$ as their theories are preserved under reduced products. In the continuous setting, this is for instance the case for atomless probability measure algebras, or (unit balls of) infinite-dimensional Hilbert spaces.

The next result mimics in the metric setting a result of Palyutin which concerns stability \cite{palyutinCategoricalHornClasses1980}. A more recent translation of it in the classical setting can be found in \cite{debondtSaturationReducedProducts2024} where it was noticed that a complete theory with the simple cover property is unstable if and only if it has the independence property, which is also true for metric structures. We refer mostly to \cite{benyaacovContinuousFirstOrder2010} for the theory of stability in the continuous framework.

\begin{definition}
	Let $T$ be a complete theory.
	
	A formula $\phi(x,y)$ is \emph{stable} if for every $\epsilon > 0$, there is no infinite sequence $(a_i,b_i : i < \omega)$ in a model of $T$ such that for all $i < j$, $|\phi(a_i,b_j) - \phi(a_j, b_i)| \geq \epsilon$. We say that $T$ is \emph{stable} if every formula is stable.
	
	A formula $\phi(x,y)$ is \emph{NIP} if there are no $r \neq s$ in $\real$ and no sequences $(a_i)_{i < \omega}$ and $(b_S)_{S \in 2^\omega}$ in a model of $T$ such that for every $i < \omega$ and $S \subseteq \omega$,
	\[\phi(a_i, b_S) = \begin{cases}
		r & \text{ if } i \in S \\
		s & \text{ if } i \not\in S\\
	\end{cases} ~.\]
	We say that $T$ is \emph{NIP} if every formula is NIP.
\end{definition}

\begin{theorem}
	Let $T$ be a complete theory which is preserved under reduced products (or equivalently, such that $T \models \SCP$). The following are equivalent :
	\begin{enumerate}[(i)]
		\item $T$ is stable;
		\item For every Palyutin formula $\phi(x,y)$ such that $\phi \geq 0$,
		\[T \models \sup_{y, z} \left( \left( \sup_x \left| \phi(x,y) - \phi(x,z) \right| \right) - \inf_x \max \left( \phi(x,y), \phi(x,z) \right) \right) \leq 0 ~;\]
		\item $T$ is NIP.
	\end{enumerate}
\end{theorem}

\begin{proof}
	(i) $\Rightarrow$ (iii) is well-known.
	
	Now, for (ii) $\Rightarrow$ (i), suppose (i) false. Consider an unstable formula $\phi(x,y)$. According to \cref{lemma:SCP B-combination cont}, we may assume that $\phi$ is a Palyutin formula. Let $M \models T$, $a_i, b_i \in M$ and $r < s$ such that $M \models \phi(a_i,b_j) \leq r$ and $M \models \phi(a_j, b_i) \geq s$ for every $i < j < \omega$. Define $\psi := \max(\phi - r, 0)$, which is a Palyutin formula and is nonnegative. Now fix $i < j < k < l < \omega$. We have $M \models \max(\psi(a_i, b_j), \psi(a_i,b_l)) = 0$ but $M \models \left| \psi(a_k, b_j) - \psi(a_k, b_l) \right| \geq s - r$, (ii) is thus false.
	
	It now suffices to prove that (iii) $\Rightarrow$ (ii). Suppose (ii) false. There exist a positive Palyutin formula $\phi$, $M \models T$, $a, b, u, v \in M$ and $\delta > 0$ such that :
	\begin{equation*}
	\left| \phi(u, a) - \phi(u, b) \right| > \max \left( \phi(v, a), \phi(v, b) \right) + \delta
	\end{equation*}
	We may assume that $\phi(u, b) > \phi(u, a) \geq 0$. Thus, since $\max(\phi(v, a), \phi(v, b)) \geq 0$, we get :
	\begin{equation}\label{eq:stability}
	\phi(u, b) > \max \left(\phi(u, a) ,\phi(v, a), \phi(v, b) \right) + \delta
	\end{equation}
	Now define for every $k < \omega$ and $S \in 2^\omega$, $c_k, d_S \in M^\omega$ such that $c_k(k) = u$, $c_k(i) = v$ if $i \neq k$, $d_S(i) = a b$ if $i \not\in S$ and $d_S(i) = a a$ if $i \in S$. Now define $\psi(x, y, z) := \max( \phi(x, y), \phi(x, z))$. 
	It is easy to check that for every $k < \omega$ and $S \in 2^\omega$ :
	\begin{align*}
		\psi(c_k, d_S) & \geq \max \left( \phi(u, a), \phi(u, b), \phi(v, a), \phi(v, b)\right) & \text{if $k \not\in S$}\\
		\psi(c_k, d_S) & \leq \max \left(\phi(u, a), \phi(v, a), \phi(v, b) \right) & \text{if $k \in S$}
	\end{align*}
	since $\psi(c_k(k), d_S(k)) = \max \left( \phi(u, a), \phi(u, b)\right)$ if $k \not\in S$ and $\psi(c_k(k), d_S(k)) = \phi(u, a)$ if $k \in S$. Using \ref{eq:stability} we deduce that $\psi$ is not NIP, and thus $T$ is not NIP.
\end{proof}

\section{Preservation theorems}
\label{sec:preservation theorems}

Let $T$ be a (consistent) theory, and $\phi$ be a continuous formula. We already know that if $T$ is axiomatisable by Palyutin conditions, then it is preserved under reduced products and preserved under reduced roots, and that if $\phi$ is a Palyutin formula, then it is bipreserved under reduced products. We will prove that the natural converses of these facts are also true. We will also prove some similar results, characterising other kinds of theories and formulas built from Palyutin formulas. Although not stated, the results of this section also hold in the classical setting replacing in the following $\HP$-sentences, $\PP$-sentences, and $\BP$-sentences respectively by Horn combinations of Palyutin sentences, positive combinations of Palyutin sentences, and Boolean combinations of Palyutin sentences.

\begin{definition}
A sentence of the form $\max(\theta_1,\dots,\theta_n)$ is :
\begin{itemize}
	\item A \emph{Horn-Palyutin sentence} (or $\HP$-sentence for short) if each $\theta_i$ can be written as $\min(D \phi, \psi)$ where $D$ is a nonincreasing unary connective and $\phi,\psi$ are Palyutin sentences.
	\item A \emph{positive-Palyutin sentence} (or $\PP$-sentence for short) if each $\theta_i$ can be written as $\min(\phi_1, \dots, \phi_n)$ where $\phi_1,\dots,\phi_n$ are Palyutin sentences.
	\item A \emph{Boolean-Palyutin sentence} (or $\BP$-sentence for short) if each $\theta_i$ can be written as $\min(\phi_1, \dots, \phi_m, D_1 \psi_1, \dots, D_n \psi_n)$ where $D_1,\dots,D_n$ are nonincreasing unary connectives and $\phi_1, \dots, \phi_m,\psi_1,\dots,\psi_n$ are Palyutin sentences.
\end{itemize}
\end{definition}

For the rest of the section, we fix a theory $T$. We define :
\begin{align*}
	T_{\Pal} := & \left\{ \phi \leq \phi^T : \phi \text{ Palyutin sentence} \right\} \\
	T_C := & \left\{ \phi \leq \phi^T : \phi \text{ $\mathfrak{C}$-sentence} \right\} \text{ where $\mathfrak{C}$ is $\HP$, $\PP$ or $\BP$.} \\
	T_{\Pal}^\ast := & \left\{ \phi = \phi^T : \phi \text{ Palyutin sentence} \right\}
\end{align*}
where $\phi^T \defeq \sup \{ \phi^M : M \models T \}$. Note that $T \models T_{\mathfrak{C}}$ where $\mathfrak{C}$ is $\Pal$, $\HP$, $\PP$ or $\BP$ so it is clear that these theories are consistent. For the latter, consistence comes from the following lemma.

\begin{lemma}\label{lemma:consistency T_P* cont}
	There exists a cartesian product $M$ of models of $T$ such that $M \models T_{\Pal}^\ast$.
\end{lemma}

\begin{proof}
	For every Palyutin sentence $\phi$ and $\epsilon > 0$, take $M_\phi^\epsilon \models T$ such that $M_\phi^\epsilon \models \phi > \phi^T - \epsilon$. Note that we also have for every Palyutin sentence $\psi$, $M_\phi^\epsilon \models \psi \leq \psi^T$. Let $M$ be the cartesian product of all the $M^\epsilon_\phi$'s. For every Palyutin sentence $\psi$ :
	\begin{equation*}
	\psi^M = \sup_{\phi, \epsilon} M^\epsilon_\phi = \psi^T \qedhere
	\end{equation*}
\end{proof}

\begin{theorem}\label{thm:char T_Pal}
	Let $M$ be an $\signature$-structure. The following are equivalent :
	\begin{enumerate}[(i)]
		\item $M \models T_{\Pal}$.
		\item There exist some filters $\filter, \filterA$, some family $N$ of models of $T$ and an $\signature$-structure $M_0$ such that $M^{\filter} \times M_0 \simeq N_{\filterA}$.
	\end{enumerate}
\end{theorem}

\begin{proof}
	Assume (i). Let $M_0$ be as in \cref{lemma:consistency T_P* cont}. We have $M \times M_0 \equiv_{\Pal} M_0$ since $\phi^M \leq \phi^{M_0}$ for every Palyutin sentence $\phi$. We now use \ref{lemma:Palyutin equivalence cont} to find $\filter$ such that $M^{\filter} \times M_0^{\filter} \simeq M_0^{\filter}$ and conclude.
	
	Assume (ii). Then, for every Palyutin sentence $\phi$ :
	\begin{equation*}
	\phi^M = \phi^{M^{\filter}} \leq \phi^{M^{\filter} \times M_0} = \limsup_{i \to \filterA} \phi^{N_i} \leq \phi^T \qedhere
	\end{equation*}
\end{proof}

\begin{corollary}\label{thm:char Palyutin theories}
	Let $T$ be a theory. The following are equivalent :
	\begin{enumerate}[(i)]
		\item $T$ is preserved under reduced products, reduced roots and (finite) cartesian factors.
		\item $T$ is axiomatizable by a Palyutin theory.
	\end{enumerate}
\end{corollary}

\begin{proof}
	Assume (i). Then by \ref{thm:char T_Pal}, we have that $T_{\Pal} \models T$, thus $T_{\Pal}$ axiomatises $T$. The other implication is just \ref{thm:pres Palyutin theories}.
\end{proof}

\begin{corollary}\label{thm:char Palyutin formulas}
	Let $\phi(x)$ be a formula. The following are equivalent :
	\begin{enumerate}[(i)]
		\item $\phi$ is bipreserved under reduced products.
		\item $\phi$ is approximable by Palyutin sentences.
	\end{enumerate}
\end{corollary}

\begin{proof}
	Assume (i). Pick $\epsilon > 0$, and $r_0 < r_1 < \dots < r_k$ such that $|r_{i + 1} - r_i| < \epsilon$ for every $0 \leq i \leq k-1$ and $\phi$ is bounded by $r_0$ and $r_k$. From \ref{thm:char Palyutin theories} we know that for every $0 \leq i \leq k$, there exists a Palyutin theory $T_i$ equivalent to $(\phi \leq r_i)$. Using the Compactness Theorem and closure of Palyutin sentences by $\max$ and increasing unary connectives, there exist some Palyutin sentence $\phi_i$ and $\lambda_i > 0$ such that $\phi \leq r_i \models \phi_i \leq 0$ and $\phi_i < \lambda_i \models \phi < r_{i + 1}$ for every $0 \leq i < k$. We can now use \ref{lemma:approximation by B-combination} to conclude.
	 
	 (ii) $\Longrightarrow$ (i) is a direct consequence \ref{thm:preservation of Palyutin formulas}.
\end{proof}

We thus proved that Palyutin sentences as we defined them form a dense set of those continuous formulas which are preserved under reduced products in both directions. In the following, we obtain other results of the same flavour concerning the other types of formula we introduced.

\begin{theorem}\label{thm:char HP formulas}
	Let $M$ be an $\signature$-structure. The following are equivalent :
	\begin{enumerate}[(i)]
		\item $M \models T_{\HP}$.
		\item Some reduced power of $M$ is isomorphic to a reduced product of models of $T$.
	\end{enumerate}
\end{theorem}

\begin{proof}
	Assume (i). For every Palyutin sentences $\phi, \psi$ and $\epsilon, \delta > 0$, we have
	\[M \models \min(\phi^M - \phi + \epsilon, \psi - \psi^M + \delta) > 0~.\]
	Hence, by assumption there exists some $N_{\phi, \psi}^{\epsilon, \delta} \models T$ such that
	\[N_{\phi, \psi}^{\epsilon, \delta} \models \min(\phi^M - \phi + \epsilon, \psi - \psi^M + \delta) > 0 ~.\]
	Fix now $\psi$ and $\delta$. We build some ultrafilter $\ultrafilter_\psi^\delta$ such that for every sentence $\theta$ :
	\begin{equation}\label{eq:B_psi^delta}
		B_\psi^\delta := \prod_{\phi, \epsilon} N_{\phi, \psi}^{\epsilon, \delta} / \ultrafilter^\delta_\psi \models \theta \leq \theta^M
	\end{equation}
	For this, we consider the family $E$ of sets of the form $\llbracket \theta < \theta^M + \eta \rrbracket$ where $\theta$ is a Palyutin sentence and $\eta > 0$ (the set of indices is taken for the family $N_{., \psi}^{., \delta}$). $E$ has the finite intersection property because Palyutin formulas are closed under $\max$. $E$ does not contain $\emptyset$ because $N_{\theta, \psi}^{\eta, \phi} \in \llbracket \theta < \theta^M + \eta \rrbracket$ by definition. Thus, we can find an ultrafilter $\ultrafilter^\delta_\psi \supseteq E$ and it is easy to check that \ref{eq:B_psi^delta} holds. Since $N_{\phi, \psi}^{\epsilon, \delta} \models \psi > \psi^M - \delta$, we also have $B_\psi^\delta \models \psi \geq \psi^M - \delta$. Now define $A$ as the cartesian product of all the $B_\psi^\delta$'s. Fix a Palyutin formula $\theta$. For $\delta > 0$, $\theta^A \geq \theta^{B^\delta_\theta} \geq \theta^M - \delta$ so $\theta^A \geq \theta^M$. Moreover, we know that $\theta \leq \theta^M$ for all the $B_\psi^\delta$'s so $\theta^A \leq \theta^M$ since $\theta$ is Palyutin. We now conclude using \ref{lemma:Palyutin equivalence cont} that some reduced power of $M$ is isomorphic to some reduced power of $A$.
	
	Assume (ii). Consider $\filter, \filterA$ and a family $(N_i)_{i \in I}$ of models of $T$ such that $M^\filter \simeq N_\filterA$. For every $\HP$-sentence $\phi$ :
	\begin{equation*}
		\phi^M = \phi^{M^\filter} = \phi^{N_\filterA} \leq \limsup_{i \to \filterA} \phi^{N_i} \leq \phi^T \qedhere
	\end{equation*}
\end{proof}

The cases of $\PP$ and $\BP$-theories and sentences can be treated the same way.

\begin{theorem}\label{thm:char PP formulas}
	Let $M$ be an $\signature$-structure. The following are equivalent :
	\begin{enumerate}[(i)]
		\item $M \models T_{\PP}$.
		\item There exist filters $\filter, \filterA$, an ultrafilter $\ultrafilter$, a family $N = (N_i)_{i \in I}$ of models of $T$ and some structure $M_0$ such that $M^\filter \times M_0 \simeq (N_\ultrafilter)^\filterA$.
	\end{enumerate}
\end{theorem}

\begin{theorem}\label{thm:char BP formulas}
	Let $M$ be an $\signature$-structure. The following are equivalent :
	\begin{enumerate}[(i)]
		\item $M \models T_{\BP}$.
		\item There exist filters $\filter, \filterA$, an ultrafilter $\ultrafilter$ and a family $N = (N_i)_{i \in I}$ of models of $T$ such that $M^\filter \simeq (N_\ultrafilter)^\filterA$.
	\end{enumerate}
\end{theorem}

As a corollary of \ref{fact:characterization of T_H}, \ref{thm:char HP formulas}, \ref{thm:char PP formulas} and \ref{thm:char BP formulas} we obtain semantic characterisations of the theories axiomatisable by each kind of formula. These results as well as \ref{thm:char Palyutin theories} are summed up in \ref{table:axiomatisability of continuous theories}.

\begin{table}[h]
\centering
\begin{tabular}{|| l | l ||}
\hline
 Axiomatisable by a ... theory & Preservation properties \\ [0.5ex] 
 \hline\hline
 Horn & Reduced products \\
 \hline
 Palyutin & Reduced products, reduced roots, and direct factors \\
 \hline
 Horn-Palyutin & Reduced products and reduced roots \\
 \hline
 Positive-Palyutin & Reduced powers, reduced roots and direct factors \\
 \hline
 Boolean-Palyutin & Reduced powers and reduced roots \\
 \hline
\end{tabular}
\caption{Axiomatisability of continuous theories characterised by their preservation properties}
\label{table:axiomatisability of continuous theories}
\end{table}

Finally, using compactness arguments in the same spirit as in the proofs of \ref{thm:characterization of Horn sentences} and \ref{thm:char Palyutin formulas}, one gets semantic characterisations of formulas axiomatised by $\HP$, $\PP$ and $\BP$ formulas which are summarised in \ref{table:approximability of continuous formulas}.

\begin{table}[h]
\centering
\begin{tabular}{||l | l ||} 
 \hline
 Approximable by ... sentences & Preservation properties \\ [0.5ex] 
 \hline\hline
 Horn & Reduced products \\
 \hline
 Palyutin & Reduced products and reduced factors \\
 \hline
 Horn-Palyutin & Reduced products and reduced roots \\
 \hline
 Positive-Palyutin & Reduced powers and reduced factors \\
 \hline
 Boolean-Palyutin & Reduced powers and reduced roots \\
 \hline
\end{tabular}
\caption{Approximability of continuous formulas characterised by their preservation properties}
\label{table:approximability of continuous formulas}
\end{table}

\bibliographystyle{alpha}
\bibliography{biblio}

\end{document}